\documentclass[11pt]{amsart}

\usepackage{amsmath}
\usepackage{amsthm}
\usepackage{amssymb}
\usepackage{amscd}          			
\usepackage{bbm}            			
\usepackage{mathrsfs}           		
\usepackage{mathalfa}

\usepackage{verbatim}           		
\usepackage{enumitem}      		

\theoremstyle{plain}
\newtheorem{Theorem}{Theorem}

\newtheorem{theorem}[Theorem]{Theorem}
\newtheorem{proposition}[Theorem]{Proposition}

\newtheorem{lemma}[Theorem]{Lemma}

\theoremstyle{definition}
\newtheorem{Definition}[Theorem]{Definition}

\newtheorem{example}[Theorem]{Example}
\newtheorem{definition}[Theorem]{Definition}
\newtheorem{remark}[Theorem]{Remark}

\theoremstyle{remark}

\newcommand{\N}{\mathbb{N}}     
\newcommand{\R}{\mathbb{R}}     
 
\def\r{\R}

\def\nuplu{\nu^{+}}       
\def\numin{\nu^{-}}         
\def\laplu{\la^{+}}  
\def\lamin{\la^{-}}  

\newcommand{\calA}{\mathscr{A}}

\newcommand{\calC}{\mathscr{C}}

\newcommand{\calK}{\mathscr{K}}

\newcommand{\calO}{\mathscr{O}}
\newcommand{\calP}{\mathscr{P}}

\DeclareMathOperator{\cl}{cl}				

\newcommand{\TM}{{TM}} 
\newcommand{\M}{{M}} 

\newcommand{\DTM}{{DTM}} 

\newcommand{\ox}{\calO(X)}
\newcommand{\cx}{\calC(X)}
\newcommand{\kx}{\calK(X)}
\newcommand{\kf}{\calK(F)}
\newcommand{\of}{\calO(F)}
\newcommand{\kv}{\calK(V)}
\newcommand{\ov}{\calO(V)}
\newcommand{\cv}{\calC(V)}
\newcommand{\ax}{\calA(X)}

\newcommand{\bccx}{\calK_{c}(X)}

\newcommand{\bcx}{\kx}

\newcommand{\bcsx}{\calK_{s}(X)}

\newcommand{\bcox}{\calK_{0}(X)}

\def\E{{{\mathcal{E}}}}

\def\D{{{\mathcal{D}}}}

\def\b{{{\mathcal{B}}}}
\renewcommand{\O}{\emptyset}

\def\D{\mathcal{D}}

\def\m{{{\mathcal{M}}}}

\def\sm{\setminus}
\def\cl{\overline}

\def\se{\subseteq}
\def\sc{\sqcup}
\def\bsc{\bigsqcup}

\def\eps{\epsilon}

\def\la1{\lambda_1}
\def\la2{\lambda_2}
\def\la0{\lambda_{0}}
\def\la{\lambda}

\usepackage{times}

\begin{document}
\title{Deficient topological measures on locally compact spaces }
\author{S. V. Butler, University of California Santa Barbara} 

\begin{abstract}
Topological measures and quasi-linear functionals generalize measures and linear functionals.
We define and study deficient topological measures on locally compact spaces. 
A deficient topological measure on a locally compact space is a set function on open and closed subsets which is finitely additive 
on compact sets, inner regular on open sets, and outer regular on closed sets. 
Deficient topological measures generalize measures and topological measures.
First we investigate positive, negative, and total variation of 
a signed set function that is only assumed to be  finitely additive on compact sets.  
These positive, negative, and total variations turn out  to be deficient topological measures.
Then we examine finite additivity, superadditivity,
smoothness, and other properties of deficient topological measures. We obtain methods for generating new 
deficient topological measures. We provide necessary and sufficient conditions 
for a deficient topological measure to be a topological measure and to be a measure. 
The results presented are necessary for further study of topological measures, deficient topological measures,
and corresponding non-linear functionals on locally compact spaces.   
\end{abstract}

\maketitle

\section{Introduction}

This paper is one in a series by the author devoted to the study of quasi-linear functionals 
and topological measures on locally compact spaces.

Topological measures (initially called quasi-measures) were introduced by J. F. Aarnes  in \cite{Aarnes:TheFirstPaper}.
These generalizations of measures and corresponding generalizations of linear functionals are connected to the problem 
of linearity of the expectational functional on the algebra of observables in quantum mechanics. 
Despite the fact that topological measures lack many features of measures, such as algebraic structure of the domain, subadditivity, etc., 
many results typical for measures still hold for topological measures.
There are many papers now devoted to topological measures and quasi-linear functionals. 
In the last 12 years numerous applications of quasi-linear functionals and topological measures to symplectic topology have been found. 
In fact, the first paper \cite{EntovPolterovich} on the connections between these two fields has been cited over 100 times, 
and was followed by many articles and a chapter in a monograph \cite{PoltRosenBook}.
Deficient topological measures were first defined and used by A. Rustad and O. Johansen in \cite{OrjanAlf:CostrPropQlf}. 
They were later independently rediscovered  by M. Svistula in \cite{Svistula:Signed}  
(where they were called  "functions of class $\Psi$") and further developed in \cite{Svistula:DTM}.  
In all previous works, deficient topological measures were defined as real-valued functions on a compact space. In this paper
we study deficient topological measures on locally compact spaces as functions into extended real numbers. 

On locally compact spaces deficient topological measures are set functions defined on open and closed subsets of a space, 
that are finitely additive on compact sets, inner compact regular on open sets, and outer regular for closed sets. 
Thus, they generalize topological measures and measures. 
It is interesting that many properties of topological measures and even measures still hold for 
deficient topological measures. 
Many results for topological measures and deficient topological measures on compact spaces remain 
valid for deficient topological measures 
on locally compact spaces. This is remarkable, since in the compact setting  we work with
open sets and closed sets, and use the fact that they are complements of each other; while in the locally compact setting 
the main sets used are compact sets and open sets. Finally, results transfer nicely from real-valued functions in a compact
setting to functions into extended reals in a locally compact setting. In this paper we generalize   
for locally compact spaces many existing results 
and obtain some new ones. The results of this paper are necessary for further study of topological measures 
and corresponding non-linear functionals on locally compact spaces.  

The paper is organized as follows. In Section \ref{PNTvari} we define the positive, negative, and total variation of 
a signed set function that is finitely additive on compact sets. We then study various properties of positive, 
negative, and total variation.
In Section \ref{snDTM} we define a deficient topological measure on a locally compact space. The positive variation is the unique
smallest deficient topological measure that on compact sets is larger than its defining signed set function. 
We study finite additivity, superadditivity,
smoothness, and other properties of deficient topological measures. 
The main focus of Section \ref{MTMDTM} 
is finding conditions under which a deficient topological measure is a topological measure or a measure. 
In Section \ref{SnewDTM} 
we discuss obtaining new deficient topological measures
by restricting and extending, and also by using continuous functions. 
Section \ref{Examples} is devoted to examples.
 
In this paper $X$ is a locally compact, connected space.
By a component of a set we always mean a connected component. We denote by $\cl E$ the closure of a set $E$.
A set $A \subseteq X$ is called \emph{bounded} if $\cl A$ is compact. 
A set $A$ is called \emph{solid} if $A$ is  connected, and $X \sm A$ has only unbounded connected components.
We denote by $ \bsc$ a union of disjoint sets.
When we consider set functions into extended real numbers, we consider functions that are not identically $\infty$ or $-\infty$. 

Several collections of sets will be used often.   They include:
$\ox$,  the collection of open subsets of   $X $;
$\cx$  the collection of closed subsets of   $X $;
$\kx$  the collection of compact subsets of   $X $;
$ \ax = \cx \cup \ox$. By $\bcox$ we denote the collection of finite unions of disjoint compact connected sets, 
by $\bccx $ the collection of compact connected sets, and by $ \bcsx$ the collection of compact solid sets.
$\calP (X)$ is the power set of $X$.

\begin{definition} \label{MDe2}
Let $X$ be a  topological space and $\mu$ be a nonnegative set function on $\E$, a family of subsets of $X$ that contains $\ox \cup \cx$. 
We say that 
\begin{itemize}
\item
$\mu$ is inner regular (or inner compact regular)  at $A \in \E$
if $\mu(A) = \sup \{ \mu(C) : \ \ C \se A, \ \ C \in  \kx \}$
\item
$\mu$ is outer regular at $A \in \E$ if 
$\mu(A) = \inf \{ \mu(U) : \ \ A \se U, \ \ U \in  \ox \}$ 
\item
$\mu$ is inner regular
if $\mu$ is inner regular at $A$ for every $A \in \E$
\item
$\mu$ is outer regular if $\mu$ is outer regular  at $A$ for every $A \in \E$
\item
$\mu$ is regular if $\mu$ is both inner and outer regular
\item
$\mu$ is compact-finite if $\mu(K) < \infty$ for any $ K \in \kx$.
\item
$\mu$ is $\tau-$ smooth on compact sets if 
$K_\alpha \searrow K, K_\alpha, K \in \kx$ implies $\mu(K_\alpha) \rightarrow \mu(K)$.
\item
$\mu$ is $\tau-$ smooth on open sets if 
$U_\alpha \nearrow U, U_\alpha, U \in \ox$ implies $\mu(U_\alpha) \rightarrow \mu(U)$.
\item
$ \mu$ is simple if it only assumes values $0$ and $1$.
\end{itemize}
\end{definition}

We consider set functions that are not identically $ \infty$ or $- \infty$.

\begin{definition}
A Radon measure  $m$  on $X$ is  a Borel measure that is compact-finite, 
outer regular on all Borel sets, and inner regular on all open sets, i.e.
for every Borel set $E$
\[ m(E) = \inf \{ m(U): \ \ E \se U, \ U \text{  open  } \}, \]
and for every open set $U$ 
\[ m(U) = \sup \{  m(K): \ \ K \se U, \  K  \text{  compact  } \}. \]
\end{definition}

\begin{remark} \label{netsSETS}
Here is a useful observation which follows, for example, from Corollary 3.1.5 in \cite{Engelking}. 
\begin{itemize}
\item[(i)]
If $K_\alpha \searrow K, K \se U,$ where $U \in \ox,\  K, K_\alpha \in \cx$, and $K$ and at least one of $K_\alpha$ are compact, 
then there exists $\alpha_0$ such that
$ K_\alpha \se U$ for all $\alpha \ge \alpha_0$.
\item[(ii)]
If $U_\alpha \nearrow U, K \se U,$ where $K \in \kx, \ U, \ U_\alpha \in \ox$ then there exists $\alpha_0$ such that
$ K \se U_\alpha$ for all $\alpha \ge \alpha_0$.
\end{itemize}
\end{remark}

The following lemma can be found, for example, in \cite{Halmos} (see Chapter X, par. 50, Theorem A).
\begin{lemma} \label{HalmEzLe}
Let $X$ be locally compact.
If $ C \se U \cup V$, where $C$ is compact and $U, V$ are open, then there exist compact sets 
$K$ and $D$ such that $C = K \cup  D, \ K \se U, \ \ D \se V$.
\end{lemma}

We would like to note the following fact. (See, for example, \cite{Dugundji}, Chapter XI, 6.2)

\begin{lemma} \label{easyLeLC}
Let $K \subseteq U, \ K \in \bcx,  \ U \in \ox$ in a locally compact space $X$.
Then there exists a bounded open sets $V$ such that 
$$ K \se V \se \cl V \se U. $$ 
\end{lemma}

We will also need the following two results (see, for example, section 2 in \cite{Butler:TMLCconstr}).

\begin{lemma} \label{LeConLC}
Let $K \subseteq U, \ K \in \bcx,  \ U \in \ox$ in a locally compact, locally connected space $X$.
If either $K$ or $U$ is connected there exist a bounded open connected set  $V$  and a compact connected set $C$
such that 
$$ K \se V \se C \se U. $$ 
One may take $C = \cl V$.
\end{lemma}

\begin{lemma} \label{LeCCoU}
Let $X$ be a locally compact and locally connected space. Suppose $K \se U, \ K \in \bcx, \ U \in \ox$. 
Then there exists $C \in \bcox$ such that $ K \se C \se U$.
\end{lemma}

\section{Positive, negative and total variation for a signed set function} \label{PNTvari}

\begin{definition} \label{laplu}
Given signed set function $\la: \kx  \longrightarrow [-\infty, \infty] $ which assumes at most one of $ \infty, -\infty$
we define two set functions on $\ox \cup \cx$, 
the positive variation $\laplu$ and the total variation $| \la|$,  
as follows:  \\
for an open subset $U \se X$ let 
\begin{eqnarray} \label{modla open set}
\laplu(U) = \sup \{\la(K): \  K \se U,  K \in \kx \}; 
\end{eqnarray}
\begin{eqnarray} \label{modnu open set}
|\la| (U) = \sup \{ \sum_{i=1}^n |\la(K_i)| : \  \bsc_{i=1}^n K_i \se U, \ K_i \se \kx,  \, n \in \N \}; 
\end{eqnarray}
and for a closed subset $F \se X$ let 
\begin{eqnarray} \label{modla closed set}
\laplu(F)  = \inf\{ \laplu (U) : \ F \se U, \ U \in \ox\};
\end{eqnarray}
\begin{eqnarray} \label{modnu closed set}
|\la| (F)  = \inf\{ |\la| (U) : \ F \se U, \ U \in \ox\}. 
\end{eqnarray} 
We define the negative variation $\lamin$ associated with a signed set function $\la$  
as a set function $\lamin = (- \la)^{+}$.  
\end{definition}


\begin{remark} \label {absla}
Let $\la, \nu$ be two signed set functions as in Definition \ref{laplu}.
We see that $| \la | = | - \la|$,
$ \laplu(U) \le  |\la| (U) $, and $  \lamin(U) \le  |\la| (U) $ for any open set $U$. Thus, $  \laplu \le  |\la|  $ and 
$ \lamin \le  |\la|  $. 

For any open set $U$ we have $| \la \pm \nu | (U)  \le |\la| (U)  + | \nu| (U)$. It follows that 
$| \la \pm \nu | \le  |\la| + | \nu|$. Similarly, $ | b \la| = | b|  | \la|$ for any real number $b$. 

If $\la \le \nu$ then $\laplu \le \nuplu$.  
\end{remark} 

\begin{lemma} \label{lapluBasic}
Let $X$ be locally compact.
Let $\la: \kx \rightarrow  [-\infty, \infty] $ which assumes at most one of $ \infty, -\infty$
be a signed set function which is finitely additive on compact sets. 
Then
\begin{enumerate}[label=(s\arabic*),ref=(s\arabic*)]
\item \label{emptset}
$\la(\O) = \laplu(\O) = \lamin(\O) = |\la| (\O) = 0.$ Also, $\laplu \ge 0$ and $| \la| \ge 0$.
\item \label{s2}
For an open set $U$, if $\laplu(U) = 0$ then $|\la |(U) =  \lamin(U)$, and 
if $\lamin(U) = 0$ then $|\la |(U) =  \laplu(U)$.
\item \label{laplmon}
$\laplu$ and $| \la|$ are monotone, i.e. if $A \se B, \ A,B \in \ox\cup \cx$ then 
$\laplu(A) \le \laplu(B)$ and $|\la|(A) \le |\la| (B)$.
\item \label{lapluad1}
$\laplu$ and $|\la|$ are finitely additive on open sets. 
\item \label{s5}
$\la(K) \le \laplu(K) $ and $\la(K) \le | \la(K)| \le | \la| (K) $ for any compact $K$. In particular, if $| \la| = 0$ on 
$\kx$, then $\la = 0$.
Also, $| \la(K) | \le \laplu(K) + \lamin(K)$.
\item \label{laplusupad1}
If $ F, C_1, \ldots, C_n$ are disjoint, $F$ is closed and all  $C_i$ are compact, then 
$$ \laplu(F \sc C_1  \sc \ldots \sc C_n) \ge \laplu(F)  + \laplu(C_1)  + \ldots +  \laplu(C_n)$$
and 
$$|\la| (F \sc C_1  \sc \ldots \sc C_n) \ge |\la| (F)  +| \la|(C_1)  + \ldots +  |\la|(C_n).$$
\item \label{laplureg}
$\laplu$ and $|\la|$ are  inner compact regular, i.e. 
$\laplu(U) = \sup \{\laplu(K): \  K \se U,  K \in \cx \}$  and 
$|\la| (U) = \sup \{| \la| (K): \  K \se U,  K \in \cx \}$ for any $U \in \ox.$ 
\item \label{comclosad}
If $ F, C_1, \ldots, C_n$ are disjoint, $F$ is closed and all  $C_i$ are compact, then 
$$\laplu(F \sc C_1  \sc \ldots \sc C_n) = \laplu(F)  + \laplu(C_1)  + \ldots +  \laplu(C_n)$$
and 
$$|\la| (F \sc C_1  \sc \ldots \sc C_n) = |\la| (F)  +| \la|(C_1)  + \ldots +  |\la|(C_n).$$
In particular, $\laplu$ and $|\la|$ are finitely additive on compact sets.
\item \label{laplusupad}
$\laplu$ and  $|\la|$  are superadditive, i.e. if $ \bsc_{t \in T} A_t \subseteq A, $  where $A_t, A \in \ox \cup \cx$,  
and at most one of the closed sets is not compact,
then 
$$\laplu(A) \ge \sum_{t \in T } \laplu(A_t) \mbox{     and     } |\la|(A) \ge \sum_{t \in T } |\la| (A_t).$$  
In particular, if $ C \se U$, where $C$ is compact and $U$ is open, then
$$\laplu(U \sm C) +\laplu(C) \le \laplu(U)  \mbox{     and     } |\la| (U \sm C)+ |\la| (C) \le |\la| (U) .$$
\item \label{s10}
$| \la | \le \laplu + \lamin.$  
\item
$(\laplu)^{+} = \laplu$.
\end{enumerate}
\end{lemma}

\begin{proof} 
\begin{enumerate}
\item[(s1)]
Since $ \la$ is not identically $\infty$ (or $-\infty$), from finite additivity of $\la$ on compact sets it follows that $|\la(\O)| < \infty$. Then $\la(\O) = 0$,
and the remaining statements easily follow. 
Note that $\laplu$ is non-negative, since $\laplu(\O) = 0$.
\item[(s2)] 
If $\laplu(U) = 0$ for an open set $U$, then $\la(K) \le 0$ 
 for any compact $K \se U$, and then by finite additivity of $\la$ on compact sets we have: 
 \begin{align*}
 |\la| (U) &= \sup \{ \sum_{i=1}^n | \la(K_i)| : \ \  \bsc_{i=1}^n K_i \se U, \  n \in \N \}  \\
 &=  \sup \{ \sum_{i=1}^n  - \la(K_i)  : \ \  \bsc_{i=1}^n K_i \se U,   \  n \in \N \} \\
 &= \sup \{ - \la(\bsc_{i=1}^n K_i)|  : \ \  \bsc_{i=1}^n K_i \se U,   \  n \in \N \} = \lamin(U). 
\end{align*}
Therefore, if $\laplu(U) = 0$ then $|\la |(U) =  \lamin(U)$, and 
if $\lamin(U) = 0$ then $|\la |(U) =  \laplu(U)$.
\item[(s3)]
Let $U, V \in \ox, \ \ F,G \in \cx$. Monotonicity of $\laplu $ and $|\la|$ is clear in the case $U \se V$, $ F \se U$ and then it is 
also easy for the cases $F \se G$, and $ U \se F$.
\item[(s4)]
Let $U_1 , U_2 \in \ox$ be disjoint. For any $C_1, C_2 \in \kx$ 
with $ C_i \se U_i, \ i=1,2$ we have by additivity of $\la$ on compact sets and  
Definition \ref{laplu}:
$$ \la(C_1)  + \la(C_2) = \la(C_1 \sc C_2) \le \laplu(U_1 \sc U_2).$$
Then by Definition  \ref{laplu} 
$$ \laplu(U_1) + \laplu(U_2) \le \laplu(U_1 \sc U_2).$$ 
For the converse inequality, note that given compact set $ C \se U_1 \sc U_2$ we 
have $C = C_1 \sc C_2$, where $C_i  = C \sm U_j = C \cap U_i \in \kx, \ i,j =1,2, \ \ i \neq j$. Then
$ \la(C) = \la(C_1) + \la(C_2) \le \laplu(U_1)  + \laplu(U_2),$ 
giving
$$ \laplu(U_1 \sc U_2)  \le  \laplu(U_1)  + \laplu(U_2).$$
So we have finite additivity of $\laplu$ on open sets. 
If $\bsc_{i=1}^n K_i \se U_1$ and $\bsc_{j=1}^m C_j \se U_2$, where $K_i, \ \ C_j \in \kx$ then 
$$ \sum_{i=1}^n |\la(K_i)| + \sum_{j=1}^m |\la(C_j)| \le |\la| (U_1 \sc U_2),$$
which by Definition  \ref{laplu}  shows that
$$ |\la| (U_1)  + |\la| (U_2)  \le |\la| (U_1 \sc U_2).$$
On the other hand, given any $\bsc_{i=1}^n C_i \se U_1 \sc U_2$ with $C_i \in \kx$ we may write
$\bsc_{i=1}^n C_i  = \bsc_{i=1}^n K_i \sc \bsc_{i=1}^n D_i$, where $K_i, D_i \in \kx, \ \ \ K_i \se U_1, \ \ \ D_i \se U_2$.
Then
$$ \sum_{i=1}^n |\la(C_i)|  \le \sum_{i=1}^n |\la(K_i)| + \sum_{i=1}^n |\la(D_i )| \le |\la|(U_1)  + |\la|(U_2), $$
which shows that   
$$ |\la| (U_1 \sc U_2) \le  |\la|(U_1)  + |\la|(U_2).$$ 
This gives the finite additivity of $| \la| $ on open sets. 
\item[(s5)]  
Let $K$ be compact. For any open set $U$ containing $K$ we have $\la(K)  \le \laplu(U)$ and 
$\la(K) \le |\la(K)| \le |\la|(U)$. Taking infimum over all such $U$ we obtain $\la(K) \le \laplu(K)$ and
$\la(K) \le | \la(K)| \le | \la| (K)$.  Since $0 \le | \la(K)| \le | \la| (K)$,  we see that if $| \la| = 0$ on 
$\kx$, then $\la = 0$. It is also easy to see that $| \la(K) | \le \laplu(K) + \lamin(K)$: for instance,
if $\la(K) \ge 0$ then $| \la(K) | = \la(K) \le \laplu(K) \le \laplu(K)  + \lamin(K)$.
\item[(s6)]
Consider first $| \la|$. It is enough to show that 
$|\la| (F \sc C) \ge |\la| (F)  + |\la| (C) $ where $ F$ is closed and $C$ is compact.
Take any open set $U$ containing $ F \sc C$. Since $X$ is completely regular,
find disjoint open sets $W, V \se U$  such that $F \se W$ and $C \se V$. By monotonicity and finite additivity
of $| \la |$ on open sets
$$ |\la| (F)  + | \la| (C)   \le  |\la| (W)  + | \la| (V) = | \la| (W \sc V) \le | \la| (U).$$
Passing to the infimum over $U$ we get $ |\la| (F)  + |\la| (C) \le |\la| ( F \sc C). $
A similar argument gives the statement for $\laplu$. 
\item[(s7)]
Let $U$ be open. By monotonicity of $\laplu$ and $| \la|$, for any compact set $K \se U$ we  have 
$\laplu(K) \le \laplu(U)$ and $|\la|(K) \le |\la|(U)$.
If $\laplu(U) = \infty$, then there are compact sets $K_n$ such that $\la(K_n) > n, \ \ n \in \N$. Then
by part \ref{s5} $ \laplu(K_n) \ge \la(K_n) > n$, and  $ \sup \{\laplu(K): \  K \se U,  K \in \kx \} = \infty = \laplu(U) $.
Assume now that $\laplu(U) < \infty$.
By Definition \ref{laplu}, given $\eps>0$, choose compact $K \se U$ such that $\laplu(U) - \la(K) <\eps$. Then
by part \ref{s5}  
$$ \laplu(U) - \laplu(K) \le \laplu(U) - \la(K) < \eps, $$
which shows that
$$\laplu(U) = \sup \{\laplu(K): \  K \se U,  K \in \cx \}.$$
Now we shall show the inner regularity for $| \la|$. 
If $|\la|(U) =\infty$, then for each $ n \in \N$ there is a finite disjoint family of compact sets $\bsc_{i \in I_n} K_i$ such that
$\sum_{i \in I_n} | \la(K_i)|  \ge n$. Let $C_n  = \bsc_{i \in I_n} K_i$. Then $C_n \se U, \ \ C_n \in \kx$ and
by parts \ref{s5} and \ref{laplusupad1} 
$$| \la| (C_n) \ge \sum_{i \in I_n} | \la| (K_i) \ge \sum_{i \in I_n} | \la(K_i)|  \ge n$$
for each $n \in \N$, so 
$ \sup \{ | \la| (C ): \ \ C \in \kx, C \se U \} = \infty =| \la| (U)  $. 
Assume now that $|\la| (U) < \infty$. For $\eps>0$ let $\bsc_{i=1}^n K_i \se U, \ \ \ K_i \in \kx$ be such that
$| \la| (U) - \sum_{i=1}^n |\la(K_i)| < \eps$. Let $K = \bsc_{i=1}^n K_i$, so compact $K \se U$.
Since $|\la| (K)  \ge \sum_{i=1}^n | \la| (K_i)  \ge \sum_{i=1}^n | \la (K_i) |$, 
we have $| \la| (U) - |\la|(K) \le  | \la| (U) - \sum_{i=1}^n |\la(K_i)| < \eps$, showing that 
$$|\la| (U) = \sup \{| \la| (K): \  K \se U,  K \in \cx \}.$$
\item[(s8)]
It is enough to show that 
$\laplu(F \sc C) = \laplu(F)  + \laplu(C) $ where $ F$ is closed and $C$ is compact. By monotonicity of $\laplu$
we may assume that $\laplu(F), \laplu(C) < \infty$.
For $\eps>0$, since $X$ is completely regular, we may find disjoint open sets $U, V$ 
such that $ F \se U, \ \ C \se V, \ \ \laplu(U), \laplu(V) < \infty,\ \ \ 
\laplu(U) - \laplu(F) <\eps, \ \ \laplu(V) - \laplu(C) <\eps.$ 
Then using parts \ref{laplmon} and \ref{lapluad1}
$$  \laplu(F \sc C) \le \laplu(U \sc V)  = \laplu(U) + \laplu(V) \le  \laplu(F)  + \laplu(C) + 2 \eps.$$
Together with part  \ref{laplusupad1} this gives
$ \laplu(F \sc C) = \laplu(F) + \laplu(C) $. A similar argument proves the statement for $| \la|$. 
\item[(s9)]
Since $\sum_{t \in T } \la(A_t) = 
\sup \{ \sum_{t \in T'} \la(A_t) : \ T' \se T,  \ T' \mbox{  is finite } \}$, 
it is enough to assume that $T$ is finite. By inner compact regularity of $\laplu$ (respectively, of $| \la|$ )
we may take all sets $A_t$ to be disjoint closed with at most one of them not compact.
The assertion follows  from part \ref{laplusupad1} and the monotonicity of  $\laplu$ (respectively, $| \la|$).
\item[(s10)]
By part \ref{s5} $| \la(K)| \le \laplu(K) + \lamin(K)$ for any compact set $K$, so using
parts \ref{comclosad} and \ref{laplusupad} it is easy to see that 
$ | \la| (A) \le \laplu (A)+ \lamin(A)$ for any open set $A$, and then for any closed set $A$. 
Thus, $|\la| \le \laplu + \lamin$.
\item[(s11)]
By part \ref{comclosad} $\laplu$ is finitely additive on compact sets, and by part \ref{emptset}  $\laplu(\O) =0$ so we may define 
$ ( \laplu)^+$. If $U$ is open, from part \ref{laplureg} it is easy to see that
$ ( \laplu)^+(U) = \laplu(U)$. Then $( \laplu)^+ = \laplu$.
\end{enumerate} 
\end{proof} 


\begin{remark}
If $\la$ is a signed measure such that its total variation (as defined for measures) is inner regular, then  
it is not hard to show that $|\la|$ is a total variation of $\la$ restricted to $\ox \cup \cx$. 
\end{remark}

\section{Deficient topological measures} \label{snDTM}

\begin{Definition}\label{DTM}
A  deficient topological measure on a locally compact space $X$ is a set function
$\nu$ on $ \cx \cup \ox \longrightarrow [0, \infty]$ 
which is finitely additive on compact sets, inner compact regular, and 
outer regular, i.e. :
\begin{enumerate}[label=(DTM\arabic*),ref=(DTM\arabic*)]
\item \label{DTM1}
if $C \cap K = \O, \ C,K \in \kx$ then $\nu(C \sc K) = \nu(C) + \nu(K)$; 
\item \label {DTM2} 
$ \nu(U) = \sup\{ \nu(C) : \ C \se U, \ C \in \kx \} $
 for $U\in\ox$;
\item \label{DTM3} 
$ \nu(F) = \inf\{ \nu(U) : \ F \se U, \ U \in \ox \} $  for  $F \in \cx$.
\end{enumerate}
We denote by $\DTM(X)$ the collection of all deficient topological measures on $X$.
\end{Definition} 

\noindent
For a closed set $F$, $ \nu(F) = \infty$ iff $ \nu(U) = \infty$ for every open set $U$ containing $F$.

\begin{remark} \label{nuxDTM3}
Let $\nu$ be a deficient topological measure on $X$.
If $X$ is locally compact and locally connected then by Lemma \ref{LeCCoU} for each open set $U$
\[  \nu(U) = \sup\{ \nu(K) : \ K \subseteq U , \ K \in \bcox \}. \]
If $X$ is  locally compact, connected, and locally connected, then from Lemma \ref{LeConLC} 
\[  \nu(X)  = \sup\{ \nu(K) : \  K \in \bccx \}, \] 
and considering for a compact connected set $C \se X$ its solid hull $ \tilde C \in \bcsx, C \se \tilde C$ 
(see section 3 in \cite{Butler:TMLCconstr} for detail), we also obtain
\[  \nu(X)  = \sup\{ \nu(K) : \  K \in \bcsx \}. \] 
\end{remark}

\begin{remark} \label{DTMagree}
The argument as in part \ref{laplmon} of Lemma \ref{lapluBasic} shows that a deficient topological measure $\nu$ is monotone. 
Since $\nu$ is not identically $\infty$, from monotonicity and finite additivity on compact sets it follows that $\nu(\O) = 0$.
If $\nu$ and $\mu$ are deficient topological measures that agree on $\kx$, then by regularity $\nu =\mu$.
If $\nu$ and $\mu$ are deficient topological measures such that $\nu \le \mu$ on $\kx$ (or on $\ox$)  then $\nu  \le \mu$.
\end{remark}

\begin{remark} \label{DTMsum} 
It is easy to see that if $\nu$ and $\mu$ are deficient topological measures, then so is $\nu + \mu$, and so is 
$\alpha \mu$ for any $\alpha \ge 0$. 
Thus, deficient topological measures constitute a positive cone.
\end{remark}
 
In Definition \ref{DTM} we do not require that $\nu(C) < \infty$ for every compact set $C$. 
By analogy with semi-finite measures (see, for example, 
\cite{Bogachev}, 1.12.132) we define semi-finite deficient topological measures:

\begin{definition}
A deficient topological measure is semi-finite if for every set $A$ with $\mu(A) = \infty$ there exists $E  \in \ox \cup \cx$ 
such that $E \se A$ and $0 < \mu(E) < \infty$.
\end{definition}

\noindent
The deficient topological measure in part \ref{semifinDTM} of Example \ref{discrDTM} below is semi-finite.  
Examples of compact-finite deficient topological measures are provided in Example \ref{1ptRnDTM} below.

\begin{lemma} \label{nunupl}
If $\nu $ is a deficient topological measure, then $\nu = \nuplu = |\nu|$. 
A deficient topological measure is finitely additive on open sets, monotone, and superadditive,
in the sense of part \ref{laplusupad} of Lemma \ref{lapluBasic},
i.e. if $ \bsc_{t \in T} A_t \subseteq A, $  where $A_t, A \in \ox \cup \cx$,  
and at most one of the closed sets is not compact, then 
$\nu(A) \ge \sum_{t \in T } \nu(A_t)$. 
\end{lemma}

\begin{proof}
From Remark \ref{DTMagree} $ \nu(\O) = 0$, and so we may apply Lemma \ref{lapluBasic}.
From Definition \ref{laplu} we see that $ \nu = \nuplu, \, \numin = 0$ and then from part \ref{s2} of Lemma \ref{lapluBasic} 
it follows that $\nu = \nuplu = |\nu|$.
Again by Lemma \ref{lapluBasic},  any deficient topological measure is finitely additive on open sets, 
monotone, and superadditive.
\end{proof}


\begin{lemma} \label{KOfiniteaddDTM}
Let $\nu : \kx \cup \ox \longrightarrow [0, \infty]$  be a set function which is inner compact regular on open sets  
and outer regular on closed sets, i.e. 
$ \nu(U) = \sup\{ \nu(C) : \ C \se U, \ C \in \kx \} $ for each open set $U$, and 
$ \nu(F) = \inf\{ \nu(U) : \ F \se U, \ U \in \ox \} $ for each closed set $F$.
(In particular, this holds for a deficient topological measure.) 
Then 
\begin{enumerate}[label=(\roman*),ref=(\roman*)]
\item
$\nu$ is monotone.
\item
$\nu$ is $\tau$-smooth on open sets; in particular, $\nu$ is countably additive on $\ox$.
\item
When $X$ is locally compact $\nu$ is finitely additive on compact sets iff it is finitely additive on open sets. 
\end{enumerate}
\end{lemma}

\begin{proof}
\begin{enumerate}[label=(\roman*),ref=(\roman*)]
\item
Use argument as in part \ref{laplmon} of Lemma \ref{lapluBasic}.
\item
Suppose $U_s \nearrow U, \, U_s, U \in \ox$.  
Let compact $K \se U.$   By Remark \ref{netsSETS}, there is $t \in S$ such that $K \se U_s$ for all $s \ge t$.
Then $\mu(K) \le \mu(U_s) \le \mu(U)$ for all $s \ge t$, and we see from the inner regularity
(whether $ \mu(U) < \infty$ or $ \mu(U) =\infty$)  that $\mu(U_s) \nearrow \mu(U)$.
\item
Assume $\nu$ is  finitely additive on open sets. 
We need to show that
\begin{align} \label{gh1}
\nu(C \sc K )  = \nu(C) + \nu(K) 
\end{align}
for any compact sets $C, \ K$.
By monotonicity, (\ref{gh1}) holds if at least one of 
$\nu(C)= \infty$ or $\nu(K) = \infty$. So let $\nu(C) < \infty $ and $\nu(K) < \infty$. For $\eps >0$ 
pick disjoint open sets $U, V$ such that $C \se U, \ K \se V,  \nu(U) - \nu(C) < \eps,  \ \nu(V) - \nu(K) < \eps$.
Then 
$$ \nu(C \sc K)  \le \nu( U \sc V) = \nu(U) + \nu(V) < \nu(C) + \nu(K)  + 2 \eps,$$
which shows that $ \nu(C \sc K)  \le  \nu(C) + \nu(K) $.
To show the opposite inequality, note that it trivially holds when $\nu(C \sc K) = \infty$, so we assume that
$\nu(C \sc K) < \infty$.  For $\eps >0$  let $W$ be an open set such that $ C \sc K \se W$ and 
$\nu(W) -\nu(C \sc K ) < \eps$. Pick disjoint open sets $U, V$ such that $C \se U, \ K \se V, \ U \sc V \se W$. 
Now 
$$ \nu(C \sc K) + \eps  \ge \nu(W) \ge \nu( U \sc V)  = \nu(U)  + \nu(V) \ge \nu(C)  + \nu(K), $$
which gives $ \nu(C \sc K)  \ge  \nu(C) + \nu(K) $.
Thus, $ \nu(C \sc K)  = \nu(C) + \nu(K) $, 
and $\nu$ is finitely additive on compact sets.

One can prove that finite additivity on compact sets implies finite additivity on open sets in a similar manner, using the 
simple fact that if compact $K \se U_1 \sc U_2$ then $K= K_1 \sc K_2$, where $K_1 \se U_1, K_2 \se U_2$.
\end{enumerate} 
\end{proof}

\begin{proposition} \label{lapldtm}
Let $X$ be locally compact.
Suppose $\la: \kx \rightarrow  [-\infty, \infty] $ which assumes at most one of $ \infty, -\infty$
is a signed set function which is finitely additive on compact sets.
Then
\begin{enumerate}[label=(\Roman*),ref=(\Roman*)]
\item \label {Ri1}
The set functions $\laplu, \lamin, |\la|$ defined as in Definition \ref{laplu} are deficient topological measures, and 
$ |\la| \le \laplu + \lamin$. 
\item \label{Ri2}
$\laplu $  is the unique 
smallest deficient topological measure such that $ \laplu(K)  \ge \la(K)$ for every $K \in \kx$
and $\lamin $  is the unique largest deficient topological measure 
such that $ - \lamin (K)  \le \la (K) $  for every $K \in \kx$.
\item \label{Ri3}
If $\sup \{ | \la(K) | : \  K \in \kx | \} \le M < \infty$ then $\laplu, \lamin$ and $| \la| $ are real-valued deficient topological measures,
$ \laplu(X), \lamin(X) \le M, \, | \la| (X) \le 2 M$.
\end{enumerate}
\end{proposition}

\begin{proof}
\begin{enumerate}[label=(\Roman*),ref=(\Roman*)]
\item
Outer regularity follows from Definition \ref{laplu}, the rest follows from Lemma \ref{lapluBasic}. 
\item
Suppose that $\nu$ is a deficient topological measure
such that $\nu(K) \ge \la(K)$ for all $K \in \kx$. By inner compact regularity of $\nu$, we have $\nu(U) \ge \laplu(U)$ 
for any open set $U$, i.e $\nu \ge \laplu$ on $ \ox$.
Both $\nu$ and $\laplu$ are deficient topological measures, so by Remark \ref{DTMagree}
$\nu \ge \laplu$. Thus, $\laplu $  is the unique smallest deficient topological measure such that 
$ \laplu  \ge \la$  on $ \kx$. Since $\lamin = (-\la)^+$, the last assertion follows. 
\item
Clearly, $ \laplu(X), \lamin(X) < M$. By part \ref{s10} of Lemma \ref{lapluBasic} $| \la| (X) \le 2 M$.
\end{enumerate} 
\end{proof}

\begin{lemma} \label{Propdetm}
Let $X$ be locally compact and let $\nu$ be a deficient topological measure on $X$. 
Given an open  set $U$ with $\nu(U) < \infty$ and $\eps >0$ there exists $C \se U, \ C \se \kx$ 
such that for any compact or open 
$E \se U \sm C$ we have $\nu(E) < \eps$.
\end{lemma}

\begin{proof} 
For an open $U$ by regularity choose compact $C \se U$ with $\nu( U )  - \nu(C) < \eps$.
Then by monotonicity and superadditivity of $\nu$, for any compact or open $E \se U \sm C$ we have
$\nu(E)  \le \nu(U \sm C) \le \nu(U)  - \nu(C) < \eps$.
\end{proof}

\begin{proposition} \label{reg4ext}
A nonnegative set function $\la: \kx \rightarrow [0, \infty] $  has a unique extension
to a deficient topological measure $ \mu$ on $X$ if and only if
$\la$ is finitely additive on compact sets and satisfies the following condition:

For any compact $C$ with $ \la(C) < \infty$ and any $\eps > 0$ there exists an open set $U$ containing $C$ such that 
\begin{align}  \label{regusl}
K \se U, \ K \in \kx \Longrightarrow  \la(K) \le \la (C) + \eps.  
\end{align} 

In this case $\mu = \laplu$. 
\end{proposition}

\begin{proof}
Necessity and condition (\ref{regusl}) follow from the definition and monotonicity of a deficient topological measure. 
To prove sufficiency, observe that by Proposition \ref{lapldtm}, the set function $\laplu$ is the 
unique smallest deficient topological measure such that 
$\laplu \ge \la$ on $\kx$. 
If $C$ is compact with $ \la(C) = \infty$, then also $ \laplu(C) = \infty$.
For $C \in \kx$ with $ \la(C) < \infty$ and $\eps > 0$, let $U$ be the set given by condition (\ref{regusl}). 
Then $$ \laplu(C) \le \laplu(U)  = \sup \{ \la (K): \  K \se U, \ K \in \kx \} \le \la(C) + \eps,$$
which implies that  $\laplu = \la$ on $\kx$.  So $ \laplu$ is the desired extension $ \mu$. 
\end{proof}

\begin{remark}
Proposition \ref{reg4ext} is the generalization to a locally compact setting of Corollary 2 in \cite{Svistula:DTM}.
\end{remark}

\begin{lemma} \label{opaddDTM}
A deficient topological measure on a locally compact space is $\tau-$ smooth on compact sets and 
$\tau-$ smooth on open sets. In particular, a deficient topological measure is  additive on open sets. 
\end{lemma}

\begin{proof}
Let $\nu$ be a deficient topological measure.
To show the $\tau-$ smoothness on compact sets, let
$K_\alpha \searrow K, K_\alpha, K \in \kx$.
It is enough to consider the case $\nu(K) < \infty$.  For $\eps >0$ let $U$ be an open set such that 
$ K \se U, \ \nu(U) < \infty$ and $\nu(U) - \nu(K) < \eps. $ By Remark \ref{netsSETS}, there exists $\alpha_0$ such that
$K_\alpha \se U $ for all $\alpha \ge \alpha_0$. 
Then 
$\nu(K_\alpha ) \le \nu(K_{\alpha_0})  \le \nu(U) < \nu(K) + \eps$
for all $\alpha \ge \alpha_0$, and the $\tau-$ smoothness on compact sets follows.
$\tau-$ smoothness on open sets follows from Lemma \ref{KOfiniteaddDTM}. 
\end{proof}

\section{Deficient topological measures, topological measures, and measures} \label{MTMDTM}

In this section we discuss the relationship between deficient topological measures, topological measures, and measures. 

\begin{Definition}\label{TMLC}
A topological measure on $X$ is a set function
$\mu:  \cx \cup \ox \to [0,\infty]$ satisfying the following conditions:
\begin{enumerate}[label=(TM\arabic*),ref=(TM\arabic*)]
\item \label{TM1} 
if $A,B, A \sc B \in \kx \cup \ox $ then
$
\mu(A\sqcup B)=\mu(A)+\mu(B);
$
\item \label{TM2}  
$
\mu(U)=\sup\{\mu(K):K \in \bcx, \  K \se U\}
$ for $U\in\ox$;
\item \label{TM3}
$
\mu(F)=\inf\{\mu(U):U \in \ox, \ F \se U\}
$ for  $F \in \cx$.
\end{enumerate}
We denote by $\TM(X)$ the collection of all topological measures on $X$.
\end{Definition} 

When $X$ is compact  $\kx = \cx$, and we simplify Definition \ref{DTM} and Definition \ref{TMLC} accordingly.
On a compact space we consider real-valued deficient topological measures.

\begin{remark}
The immediate difference between the definitions of a topological measure and a deficient 
topological measure is that a deficient topological measure is additive on 
compact sets, while a topological 
measure is additive on $ \ox \cup \cx$. 
\end{remark}

By $\M(X)$ we denote the collection of all Borel measures on $X$ that are inner regular on open sets and 
outer regular (restricted to $\ox \cup \cx$). 

\begin{remark}
Let $X$ be locally compact. In general,
$$ \M(X) \subsetneqq  \TM(X) \subsetneqq  \DTM(X). $$
The proper inclusion $ \M(X) \subsetneqq  \TM(X)$ follows from  
examples in the last section in \cite{Butler:TMLCconstr}. 
The proper inclusion $ \TM(X) \subsetneqq  \DTM(X)$ follows from examples in the last section of this paper. 
We would like to know when  a deficient topological measure is  a measure from $\M(X)$ or a topological measure.
\end{remark}

\begin{theorem} \label{DTMtoTM}
\begin{enumerate}[label=(\Roman*),ref=(\Roman*)]
\item
Let $X$ be compact, and $\nu$ a deficient topological measure. The following are equivalent:
\begin{enumerate}
\item[(a)]
$\nu$ is a topological measure
\item[(b)]
$\nu(X) = \nu(C)  + \nu(X \sm C), \ \ \ C \in \cx$ 
\item[(c)]
$\nu(X) \le \nu(C)  + \nu(X \sm C), \ \ \ C \in \cx$  
\end{enumerate}
\item
Let $X$ be locally compact, and $\nu$ a deficient topological measure. 
The following are equivalent:
\begin{enumerate}
\item[(a)]
$\nu$ is a topological measure
\item[(b)]
$\nu(U) = \nu(C)  + \nu(U \sm C), \ \ \ C \in \kx, \ \ \ U \in \ox$ 
\item[(c)]
$\nu(U) \le \nu(C)  + \nu(U \sm C), \ \ \ C \in \kx, \ \ \ U \in \ox$
\end{enumerate}
\end{enumerate}
\end{theorem}

\begin{proof}
\begin{enumerate}[label=(\Roman*),ref=(\Roman*)]
\item
(a) $\Longrightarrow$ (b) follows from the definition of a topological measure on a compact space. 
Superadditivity of a deficient topological measure implies equivalence of (b) and (c). 
Since a deficient topological measure is finitely additive on closed 
and open sets (see Lemma \ref{KOfiniteaddDTM}), to show that (b) implies (a) it is enough to check that 
$\nu(A \sc B) = \nu (A) + \nu(B)$ when $A, B$ are not both closed
or both open. We may also assume that $A$ and $A\sqcup B$ are not
both closed or both open. Then the sets
$A$ and $X\setminus (A\sqcup B)$ \emph{are} both closed or both open, and disjoint.
By additivity on open (or closed) sets
$$ \nu(A)+\nu(X\setminus (A\sqcup B))=\nu \left(A\sqcup  [X\setminus (A\sqcup B)] \right) =
\nu(X\setminus B), $$
so by (b) 
$$ \nu(A)+\nu(X) - \nu(A \sc B)=\nu (X) - \nu(B), $$
i.e.
$$ \nu(A) + \nu(B) = \nu(A \sc B).$$
\item
(a) $\Longrightarrow$ (b) follows from the definition of a topological measure on a locally compact space. 
Superadditivity of a deficient topological measure gives equivalence of (b) and (c). Finite additivity of a deficient 
topological measure on open sets  (see Lemma \ref{KOfiniteaddDTM}) and Proposition 37 in 
\cite{Butler:TMLCconstr} show that (b) implies (a).
\end{enumerate}
\end{proof}

Now we are interested in finding out when a deficient topological measure is a measure.  

\begin{definition} \label{Defmu*}
For a deficient topological measure $\mu$ on a locally compact space $X$ define 
a set function $\mu^*: \calP (X) \longrightarrow [0, \infty]$ as follows:
$$ \mu^*(E) = \inf \{ \mu(U): \ E \se U, \ U \in \ox \}.  $$
for any subset $E$ of $X$.
\end{definition}

Note that for any open set $ \mu^*(U) = \mu(U)$.

\begin{lemma} \label{outerM}
Let $\mu$ be a deficient topological measure on a locally compact space $X$ which is finitely subadditive on open sets, 
i.e. if $U, V$ are open subsets of $X$,  then $\mu(U \cup V) \le \mu(U) + \mu(V)$. Let $\mu^*$ be as in 
Definition \ref{Defmu*}.
Then 
\begin{enumerate}[label=(\roman*),ref=(\roman*)]
\item
$\mu^*$ is an outer measure on $X$.
\item
A subset $E$ of $X$ is  $\mu^*$-measurable iff for each open set $U$.
\begin{align} \label{onlyforop}
\mu(U) \ge \mu^*(U \cap E ) + \mu^*(U \sm E).
\end{align}
\item
Each closed subset of $X$ is $\mu^*$-measurable. 
\end{enumerate}
\end{lemma}

\begin{proof}
\begin{enumerate}[label=(\roman*),ref=(\roman*)]
\item
Clearly, $\mu^*(\O) = 0, \  \mu^*$ is monotone and finitely subadditive. 
To show that $\mu^*$  is an outer measure, we only need to check that it is countably subadditive.
From finite subadditivity of $\mu$ and Lemma \ref{opaddDTM} it
is easy to see that $\mu$ is countably subadditive on open sets. 
We shall show countable additivity of $\mu^*$. By monotonicity we may assume $\nu^*(E_i) < \infty$ for each $i$.  
For arbitrary$\eps>0$ and sets 
$E_1, E_2, \ldots$ choose open sets $U_1, U_2, \ldots$ such that  $E_i \se U_i$ and 
$\mu(U_i) -\mu^*(E_i) < \eps / 2^i$ for $i=1,2, \ldots$.
Then
\begin{align*}
\mu^*(\bigcup_{i=1}^{\infty} E_i) &\le \mu(\bigcup_{i=1}^{\infty} U_i) \le 
\sum_{i=1}^{\infty} \mu (U_i) \le \sum_{i=1}^{\infty} \mu^* (E_i)  + \eps, 
\end{align*}
so $\mu^*$ is countably subadditive. 
\item
Assume  (\ref{onlyforop}) holds. 
The  proof is basically as in \cite[p.234, Theorem D]{Halmos} and proceeds as follows.
Let $A$ be an arbitrary set, and take
any open $U$ containing $A$.  Since
\begin{align*}
\mu(U) &= \mu^*(U) \ge \mu^*(U \cap E) + \mu^*(U \sm E) \\
&\ge  \mu^*(A \cap E) + \mu^*(A \sm E),
\end{align*}
taking infimum over all open $U$ containing $A$ we see that
\[ \mu^*(A) \ge  \mu^*(A \cap E) + \mu^*(A \sm E).\] 
Since $ \mu^*$ is also subadditive, $E$ is $ \mu^*$-measurable.
\item
The argument is essentially as in \cite[p.235, Theorem E]{Halmos} and proceeds as follows.
Let $F$ be closed. We need to show that  $\mu (U) = \mu^*(U)  \ge \mu^*(U \cap F) + \mu^*(U \sm F) $
for any open set $U$, 
and it is enough to consider the case $ \mu^*(U),  \mu^*(U \cap F), \mu^*(U \sm F) < \infty$.
Let $U$ be open. Let $C$ be a compact subset of an open set
$U \sm F$, and let $K$ be a compact subset of  an open set $U \sm C$.
Then $C \sc K \se U$ and  $\mu^*(U) = \mu(U) \ge \mu(C) + \mu(K)$.
Taking supremum over all compacts $K  \se U \sm C$  we see that  
\[ \mu^*(U) \ge \mu(C) + \mu(U \sm C) = \mu(C) + \mu^*(U \sm C) \ge \mu(C) + \mu^*(U \cap F). \]
Taking supremum over all compacts $C \se U \sm F$ we have:
\[ \mu^*(U) \ge  \mu(U \sm F) + \mu^*(U \cap F) =  \mu^*(U \sm F) + \mu^*(U \cap F), \]
which finishes the proof of the statement.
\end{enumerate}
\end{proof}

\begin{theorem} \label{extToMe}
Let $\mu$ be a deficient topological measure on a locally compact space $X$ which is finitely subadditive on open sets, 
i.e. if $U, V$ are open subsets of $X$,  then $\mu(U \cup V) \le \mu(U) + \mu(V)$. Then
$\mu$ admits a unique extension to an inner regular on open sets, outer regular Borel measure 
$m$ on the Borel $\sigma$-algebra of subsets of $X$. 
$m$ is a Radon measure iff $\mu$ is compact-finite. 
If $\mu$ is finite then $m$ is a regular Borel measure.
\end{theorem}

\begin{proof}
By Lemma \ref{outerM}, $\mu^*$ is an outer measure. 
The collection $\m$ of all $\mu^*$-measurable sets
is a $\sigma$-algebra, and the restriciton of $\mu^*$ to $\m$ is a measure.
By Lemma \ref{outerM}, the  $\sigma$-algebra $\m$ contains the Borel $\sigma$-algebra $\b(X)$.
Let $m$ be the restriction of $\mu^*$ to $\b(X)$. Thus, $m$ is a Borel measure. 
For each Borel set $E$
\begin{align*}
m(E) &= \mu^*(E) = \inf \{ \mu(U) : E \se U, \ U \in \ox\} \\
&= \inf \{ \mu^*(U) : E \se U, \ U \in \ox\} = \inf \{ m (U) : E \se U, \ U \in \ox\},
\end{align*}
showing that $m$ is outer regular, and this gives 
the uniqueness of extension of $\mu$ to an outer regular Borel measure on $\b(X)$.
 
For a closed set $F$  we see that  $ m(F) = \inf \{ \mu(U) : F \se U, \ U \in \ox\} = \mu(F)$. 
For an open set $U$ we have $m(U) = \mu^*(U) = \mu(U)$. 
Thus, $ \mu = m $ on $ \ox \cup \cx$. Also, $m$ is inner regular on open sets.

We see that $\mu$ is compact-finite iff $m$ is. Thus, $ m$ is a Radon measure iff $\mu$ is compact-finite.
If $\mu$ is finite, it is easy to show that outer regularity of $m$
is equivalent to inner regularity of $m$, hence, $m$ is a regular Borel measure.
\end{proof}

\begin{remark}
This proof is very similar to the proof in \cite{Halmos} that a content 
on a locally compact space extends to a compact regular Borel measure.  
The reason is that a deficient topological measure that is subadditive and finite on compact sets is a content 
in the classical definition (see \cite{Halmos}, $\S$ 53). 
The proof in the case where $X$ is compact, $\mu$ is a 
topological measure, and the extension is to the Borel algebra was first given by Wheeler in \cite{Wheeler}; 
the proof in the case where $X$ is compact and
$\mu$ is a  deficient topological measure was first given by Svistula in \cite{Svistula:Signed}.
If $\mu$ is a compact-finite deficient topological measure on a locally compact space, the existence of a unique extension
of $\mu$ to an inner regular on open sets, outer regular Borel measure 
$m$ on the Borel $\sigma$-algebra of subsets of $X$ follows also from Theorem 7.11.1 in \cite{Bogachev}, vol.2.
Note that for a locally compact space $X$ we no longer require 
the deficient topological measure to be real-valued on compacts. 
See Example \ref{discrDTM}, part \ref{semifinDTM}.
\end{remark}

We are now ready to present necessary 
and sufficient conditions for a deficient topological measure to be a measure.

\begin{theorem} \label{subaddit}
Let $\mu$ be a deficient topological measure on a locally compact space $X$. 
The following are equivalent: 
\begin{itemize}
\item[(a)]
If $C, K$ are compact subsets of $X$, then $\mu(C \cup K ) \le \mu(C) + \mu(K)$.
\item[(b)]
If $U, V$ are open subsets of $X$,  then $\mu(U \cup V) \le \mu(U) + \mu(V)$.
\item[(c)]
$\mu$ admits a unique extension to an inner regular on open sets, outer regular Borel measure 
$m$ on the Borel $\sigma$-algebra of subsets of $X$. 
$m$ is a Radon measure iff $\mu$ is compact-finite. 
If $\mu$ is finite then $m$ is a regular Borel measure.
\end{itemize}
\end{theorem}

\begin{proof}
$(a)  \Longrightarrow (b)$. 
Let $U, V$ be open in $X$, and let $D \se U \cup V, \, D \in \kx$.
By Lemma \ref{HalmEzLe} write $D= C \cup K, \, C \se U, \, K \se V$, where $C,K \in \kx$.
Since 
\[ \mu(D) = \mu(C \cup K)  \le \mu(C) + \mu(K) \le \mu(U) + \mu(V), \]
taking supremum over all compacts $D$ we have $\mu(U \cup V) \le \mu(U) + \mu(V)$. \\
$(b) \Longrightarrow (c)$. 
This is Lemma \ref{extToMe}. \\
$(c) \Longrightarrow (a)$.
$\mu(C \cup K) = m (C \cup K) = m(C) + m(K \sm C) \le  m(C) + m(K) =  \mu(C) + \mu(K)$.
Note that $K \sm C$ is neither open nor closed in general, so $\mu(K \sm C)$ need not be defined.
\end{proof}

\section{New deficient topological measures} \label{SnewDTM}

In this section we gives some methods for obtaining new deficient topological measures from known ones.

\begin{proposition}  \label{finvDTM}
Let $X$ be a compact space and $Y$ be a locally compact space.
If $\nu$ is a deficient topological measure on $X$ and $f : X \rightarrow Y$ is continuous
then the set function $ \nu \circ f^{-1}$ on $\calO(Y)\cup \calC(Y)$ is a deficient topological measure on $Y$.
If $\nu$ is a topological measure on $X$, then $ \nu \circ f^{-1}$ is a topological measure on $Y$.
\end{proposition}

\begin{proof}
First, we shall show that  $\nu \circ f^{-1}$ is inner regular.
Let $U$ be open in $Y$. If  $(\nu \circ f^{-1}) (U) = \infty$, for each $n \in \N$ find a compact in $X$ set $K_n \se  f^{-1} (U)$ such that
$\nu(K_n) > n$. Then for the set $C_n = f(K_n)$, compact in $Y$, we have: $C_n \se U, \, K_n \se f^{-1}(C_n)$, and 
$(\nu \circ f^{-1}) (C_n) \ge \nu(K_n) > n$, so 
$ \sup\{ \nu \circ f^{-1} (C) : \ C \se U, \ C \in \calK(Y) \}  = \infty = \nu \circ f^{-1} (U)$.
If  $(\nu \circ f^{-1}) (U) < \infty$, for an arbitrary $ \eps >0$  choose a compact $K \se  f^{-1} (U)$ such that
$(\nu \circ f^{-1}) (U) - \nu(K) < \eps$. For compact set $C = f(K)$ we have:
$C \se U, \, K \se f^{-1}(C)$ and $(\nu \circ f^{-1}) (U) - (\nu \circ f^{-1}) (C) \le (\nu \circ f^{-1}) (U) - \nu(K) < \eps$, showing again 
that $ \sup\{ \nu \circ f^{-1} (C) : \ C \se U, \ C \in \calK(Y) \} = \nu \circ f^{-1} (U)$.

Now we shall show the outer regularity. Let $F$ be closed in $Y$. 
By monotonicity of $ \nu$ it is enough to assume that $(\nu \circ f^{-1}) (F) < \infty$.
For $ \eps >$ choose $W \in \ox$ such that $ f^{-1}(F) \se W$ and $  \nu(W) < (\nu \circ f^{-1}) (F) + \eps$. 
The set $Y \sm f(X \sm W)$ is open in $Y$ and contains $F$.
Also, $f^{-1} (F) \se f^{-1} (Y \sm f(X \sm W))  \se W$, so 
$ (\nu \circ f^{-1}) (Y \sm f(X \sm W)) \le  \nu(W) < (\nu \circ f^{-1}) (F) + \eps$,
showing that $ \inf \{ \nu \circ f^{-1} (V) : \ F \se V, \ V \in \calO(Y)\} = (\nu \circ f^{-1}) (F)$. 

It is easy to see the finite additivity of $\nu \circ f^{-1}$ on open sets. By Lemma \ref{KOfiniteaddDTM}, $\nu \circ f^{-1}$
is then finitely additive on compact sets. Thus, $\nu \circ f^{-1}$ is a deficient topological measure on $Y$.   

If $\nu$ is a topological measure on $X$, using Lemma \ref{DTMtoTM} it is easy to see that $\nu \circ f^{-1}$ is
in fact a topological measure on $Y$. 
\end{proof}

\begin{remark}
Proposition \ref{finvDTM} holds for a locally compact space $X$ if we require in addition $f$ to be a closed mapping.
\end{remark} 


\begin{remark}
Proposition \ref{finvDTM} was first proved for both spaces compact and $ \nu$ finite in \cite{OrjanAlf:CostrPropQlf}.
\end{remark}

If $\mu$ is a deficient topological measure on $X$, and $ A \in \ox \cup \cx$, 
one can not obtain a deficient topological measure on $A$  by simply 
restricting $\mu$ to $A$, i.e., considering $\mu(A \cap B), \ B \in \ox \cup \cx$. One simple reason for this is that 
the intersection of two arbitrary sets from $ \ox \cup \kx$ does not in general belong to $ \ox \cup \kx$.   
Nevertheless, we may obtain a deficient topological measures on closed and open sets as restrictions.

Let $X$ be a locally  compact space, and $F$ be a closed subset. Consider $F$ as a space with the subspace topology.
Note that open subsets of $F$ are described as follows: 
$$ \of = \{ U \cap F: \ \ U \in \ox \},  $$
closed subsets of $F$ have the form $F \cap G$ where $ G \in \cx$, a compact in $F$ is also a compact in $X$, and so
compact subsets of $F$ are
$$ \kf = \{ C \in \kx: \ C \se F \}$$

\begin{proposition} \label{restrDTM}
Let $X$ be locally compact,  and $\nu$ be a deficient topological measure on $X$.
Let $F$ be a closed subset equipped with the subspace topology.
Then there exists a deficient topological measure $\nu_F$ on $F$ with the subspace topology 
such that for any $A \in \of \cup \kf$ 
$$ \nu_F(A) = \sup\{ \nu(K): \ K \se A, \ K \in \kf \}.$$
\end{proposition}

\begin{proof}
Define $\nu_F$ on  $\kf$ to be the restriction of $\nu$ to $\kf$. Then $\nu_F$ is finitely additive on $\kf$.
Let $C  \in \kf$ be such that $ \nu_F(C)= \nu(C) < \infty$. 
Since $\nu$ is a deficient topological measure on $X$, it satisfies condition (\ref{regusl}), so
for $\eps >0$ there exists $U \in \ox$ containing $C$ such that for any $K \se U, K \in \kx$ we have 
$\nu(K) \le \nu(C) + \eps$.
Let $V= U \cap F$. Then $V \in \of, \ C \se V$ and for any $D \in \kf$  such that $D \se V$ we have $ D \se U$ and so
$$ \nu_F(D) = \nu(D)  \le \nu(C) + \eps = \nu_F(C) + \eps.$$
Thus, condition (\ref{regusl}) holds for $\nu_F$, and by Proposition \ref{reg4ext} $ \nu_F$ extends to a required 
deficient topological measure.
\end{proof} 
 
\begin{proposition} \label{extenDTM}
Let $\nu$ be a deficient topological measure on $X$, and $ X$ be a closed subset of  $Y$. 
Define $\nu_1$ on open and closed subsets of $Y$ by letting $\nu_1(A) = \nu(A  \cap X)$. 
Then $\nu_1$ is a deficient topological measure on $Y$.
\end{proposition} 

\begin{proof}
The proof is easy and is left as an exercise.
\end{proof}

\begin{definition}
We will call the deficient topological measure $\nu_F$ in Proposition \ref{restrDTM} the restriction of $\nu$ to $F$,
and the deficient topological measure $\nu_1$ in Proposition \ref{extenDTM} the extension of $\nu$ to $Y$.
\end{definition}

We have an analog of Proposition \ref{restrDTM} for an open set.

\begin{proposition}  \label{restrDTMop}
Let $X$ be locally compact,  and $\nu$ be a deficient topological measure on $X$.
Let $V$ be an open subset equipped with the subspace topology, 
Then there exists a deficient topological measure $\nu_V$ on $V$ with the subspace topology 
such that for any $A \in \ov \cup \kv$ 
$$ \nu_V(A) = \inf \{ \nu(U): \ A \se U, \ U \in \ov \}.$$
\end{proposition}

\begin{proof}
Let $V \in \ox$. We consider the space $V$ with the  subspace topology. 
Then $\ov = \{ U \in \ox : \, U \se V\}$.
Define a set function $\nu_V$ on $\ov \cup \cv$ by $\nu_V (U) = \nu(U)$ for $ U \in \ov$ and 
$\nu_V(F) = \inf \{ \nu_V(U):  F \se U, \, U \in \ov \} =  \inf \{ \nu(U):  F \se U, \, U \in \ov \} $ for $F \in \cv$.
Then $ \nu_V$ is finitely additive on $ \ov$ and outer regular.
We need to show the inner regularity of $\nu_V$. Let $ U \in \ov$. 
Suppose first that $ \nu_V(U) = \nu(U) < \infty$. For $ \eps > 0$ pick $K \in \kx, K \se U$ such that 
$\nu(U) - \nu(K) < \eps$. By Lemma \ref{easyLeLC}  let $W$ be a bounded open set such that 
$ K \se W \se U$. Note that  $K \in \kv$ and
\[ \nu_V(U) - \nu_V(K) \le \nu_V(U) - \nu_V(W) = \nu(U) -\nu(W)  \le \nu(U) - \nu(K) < \eps. \] 
Thus, $ \nu_V(U) = \sup\{ \nu_V(K) : \, K \in \kv \}$.

Now suppose  $ \nu_V(U) = \nu(U) = \infty$.  For $n \in \N$ pick $K_n \in \kx, K \se U$ such that $ \nu(K_n) > n$.
Then $K_n \in \kv$. 
For any $W \in \ov$ such that  $ K_n \se W \se U$  we have
$\nu_V(W)  = \nu(W)  \ge \nu(K_n) > n$, and so 
$ \nu_V(K_n) = \inf \{ \nu_V(W): \, K_n \se W \se U, \, W \in \ov \} \ge n$.   
Then $\sup \{ \nu_V(K): \, K \se U, \, K \in \kv \} = \infty = \nu_V(U)$.
\end{proof} 

\begin{definition} \label{dtRestrU}
Let $X$ be locally compact, $\mu$ be a deficient topological measure on $X$.
Let $V \in \ox$. Define a set function $\mu_V$ on $\ox \cup \cx$ by letting
$$ \mu_V ( U ) = \mu(U \cap V), \ U \in \ox, $$
and 
$$ \mu_V (F) = \inf \{ \mu_V (U): \ F \se U, \ U \in \ox \} , \ F \in \cx.$$
\end{definition}

\begin{remark} \label{muVFmu}
If $ F \se V$ then  $\mu_V(F) = \mu(F) $, because
$$ \mu(F) = \inf\{ \mu(U) :\,  F \se U \se V \} = \inf \{ \mu_V(U) : \, F \se U \se V \} = \mu_V(F).$$
\end{remark}

\begin{theorem}  \label{restUdtm}
Let $X$ be locally compact.
Let $\mu$ be a deficient topological measure on $X, \ V \in \ox$. Then $ \mu_V$ defined in Definition \ref{dtRestrU} 
is a deficient topological measure on $X$.
\end{theorem}

\begin{proof}
By its definition, $\mu_V$ is outer regular. By Remark \ref{muVFmu},
if $ C \in \cx, C \se V$ then $\mu_V(C) = \mu(C)$. 
We shall show that $\mu_V$ is inner compact regular. Let $ U \in \ox$. Assume first that  $ \mu_V(U) = \mu(U \cap V) = \infty $.
For $n \in \N$ choose $ K_n \in \kx$ such that $ K_n \se U \cap V, \ \mu(K_n) \ge n$. 
Then by Remark \ref{muVFmu}  $ \mu_V(K_n) = \mu(K_n) > n$ and so 
$ \sup \{ \mu_V (K) : \ K \se U,  K \in \kx\} = \infty = \mu_V(U)$.
Now assume $\mu_V(U) = \mu(U \cap V) < \infty $. Given $ \eps >0$ 
find $ K \in \kx$ such that $ K \se U \cap V$ and  $ \mu(U \cap V) - \mu(K) < \eps$.  
Then $ \mu_V(K) = \mu(K) > \mu(U \cap V) -\eps = \mu_V(U) - \eps$, and 
the inner compact regularity follows. 
Finite additivity of $ \mu_V$ on open sets follows from the same property of $\mu$. 
By Lemma \ref{KOfiniteaddDTM} $\mu_V$ is finitely additive on compact sets. 
Hence, $\mu_V$ is a deficient topological measure.
\end{proof}

\begin{lemma} \label{muFC}
Let $X$ be locally compact,  $\mu$ be a deficient topological measure on $X$, and $F \in \cx$. Then for any $C  \in \kx$  
\begin{align} \label{muFCinf}
\mu(F \cap C) = \inf \{ \mu_V (C) : \  F \se V, \ V \in \ox \}, 
\end{align}
where $\mu_V$ is a deficient topological measure on $X$ from Theorem \ref{restUdtm}.
If $ F \in \kx$ then (\ref{muFCinf}) holds for any $C \in \cx$.
\end{lemma}

\begin{proof}
Let $ F  \se V, C \se U$, where $ U , V \in \ox$. 
Then $\mu(F \cap C) \le \mu(V \cap U) = \mu_V(U)$, 
so $ \mu(F \cap C) \le   \inf \{\mu_V(U) : C \se U, \ U \in \ox\} = \mu_V(C)  $. 
The equality is easy to see if $ \mu(F \cap C) = \infty$.

Now suppose that $\mu(F \cap C) < \infty$. Given $\eps >0$, 
let $ W \in \ox$ be such that $ F \cap C \se W$ and $ \mu(W) - \mu(F \cap C) < \eps$.
Since $ C \cap (F \sm W) = \O$, whether $F \in \cx, C \in \kx$ or $ F \in \kx, C \in \cx$,  
by complete regularity of $X$ we may choose disjoint open sets  
$ U, U_1 \in \ox$ such that $ C \se U$ and  $F \sm W \se U_1$. 
Let $V = W \cup U_1$. Then $F \se V$ and $ V \cap U \se  W \cap U  \se W$.
Now 
\[ \mu_V(C) - \mu(F \cap C) \le \mu_V(U) - \mu(F \cap C) \le \mu(W) - \mu(F \cap C) < \eps, \]
and the statement follows.
\end{proof}

\begin{theorem} \label{DTMmuF}
Let  $X$ be locally compact, and let $F \se \cx$. 
There exists a deficient topological measure $\mu_F$ on $X$ such that 
\[ \mu_F ( K )  = \mu(F \cap K) \ \ \mbox{for} \ \  K \in \kx, \]
\[ \mu_F(U) = \sup \{ \mu_F(K) : \ K \se U , \ K \in \kx \} \ \  \mbox{for}  \ \ U \in \ox. \]
If $F \in \kx$ then $\mu_F(C) = \mu(F \cap C)$ for every $ C  \in \cx$.
\end{theorem} 

\begin{proof}
The existence of required $\mu_F$ can be proved, for example, by applying Proposition \ref{extenDTM} to a deficient topological 
measure on $F$ given by Proposition \ref{restrDTM}. We shall show now that if $F$ is compact, 
then $\mu_F(C) = \mu(F \cap C)$ for every 
closed set $C$. We need to check that $\mu(F \cap C) = \inf \{ \mu_F(U) : \ C \se U , \ U \in \ox \}$.
For any open set $U$ such that $ C \se U$ we have 
\[ \mu(F \cap C) = \mu_F(F \cap C) \le \mu_F(U), \]
so 
$$\mu(F \cap C) \le \inf \{ \mu_F(U) : \ C \se U , \ U \in \ox \}.$$
It is enough to show the opposite inequality for the case $\mu(F \cap C) < \infty$. 
By Lemma \ref{muFC} $ \mu(F \cap C)= \inf \{ \mu_V (C) : \  F \se V, \ V \in \ox \}$. Given $ \eps >0$ 
choose $V \in \ox$ such that $ F \se V$ and $ \mu_V(C) - \mu(F \cap C) <\eps$.
Note that for any $U \in \ox$ with $C \se U$ we have $\mu_F(U) \le \mu_V(U)$. Then
using Definition \ref{dtRestrU}
\begin{align*}
\mu(F \cap C) + \eps &> \mu_V(C)  = \inf \{ \mu_V(U) :  \  C \se U \ U \in \ox, \} \\
& \ge \inf \{ \mu_F(U) : \  C \se U \ U \in \ox, \}. 
\end{align*}
This gives
$$\mu(F \cap C) \ge \inf \{ \mu_F(U) : \ C \se U , \ U \in \ox \},$$ 
and the proof is complete.
\end{proof}

\begin{remark}
When $X$ is compact, topological measures (deficient topological measures) that are restrictions of topological measures (deficient topological measures)  
to sets appeared in several papers, including  \cite{GrubbLaberge}, \cite{Laberge}, \cite{Svistula:DTM}, \cite{Svistula:Integrals}.
Proposition \ref{restrDTM}, Theorem \ref{restUdtm}, and Lemma \ref{muFC} are generalizations to a locally compact case 
of Proposition 3 in \cite{Svistula:DTM}, Proposition 5.1 in \cite{Svistula:Integrals}, and the stated without proof part (4) of Proposition 5.2 in \cite{Svistula:Integrals}.
\end{remark}  

\section{Examples of deficient topological measures} \label{Examples}

Proposition \ref{lapldtm} allows us to build examples of deficient topological measures from simple set functions. 
Results in section \ref{SnewDTM} allow us to obtain new deficient topological measures from existing ones. 
We will  now present some concrete examples.  

\begin{example} \label{Svconset}
Let $X$ be locally compact, and let $D$ be a connected compact subset of $X$. Define a set function 
$\nu$ on $\ox \cup \cx$  by setting $\nu(A) = 1$ if $ D \se A$ and $\nu(A) = 0$ otherwise, for any
$A \in \ox \cup \cx$. If $D \se C \sc K$, where $C, K \in \kx$ then by connectedness either $D \se C$ 
or $ D \se K$, and this implies finite additivity  of $\la$ on $\kx$. It is also easy to check the inner and outer 
regularity of $\nu$: for example, if $ F \in \cx,  \ \nu(F) = 0$, then there exists a point $ x \in D \sm F$, and 
taking $U = X \sm \{x\}$ we see that $ F \se U,  \ \nu(U) = 0$,  so outer regularity is satisfied for $F$.
By Definition \ref{DTM}, $\nu$ is a deficient topological measure. Note that if $D$ is a singleton, then 
$\nu$ is a point mass. If $D$ has more than one element, then $\nu$ is a deficient topological measure, but not 
a topological measure: consider $ K = \{ x \}$, where $x \in D$, and $U = X \sm K$. Then 
$\nu(K) + \nu(U) = 0 \neq 1 = \nu(X)$, and by Definition \ref{TMLC} $\nu$ is not a topological measure.
\end{example}

\begin{example} \label{lanolDTM}
Let $\E$ be a finite family of disjoint bounded nonempty connected subsets of $X$. 
Let $\D$ be the family of closed subsets of $X$ 
that contain at least one set from $\E$, and $\la_{0} $ be some set function on $ \D$. 
For a closed set $F$ we set

\begin{align} \label{lala0}
  \la(F)  & = 
  \left\{
  \begin{array}{rl}
   \sum_{i=1}^n \la_{0}(F_i), & \mbox{where }  F_i \mbox{ are components of } F \mbox{ that belong  to } \D \\
  0 , & \mbox{if no component of } F  \mbox{  belong to }  \D
  \end{array}
  \right.
\end{align}

In particular, $ \la(\O) = 0$. 
The components of two disjoint compact sets give the components of their union. It follows that $\la$ is finitely 
additive on $\kx$. 
\begin{enumerate}[label=(\roman*),ref=(\roman*)]
\item \label{exA}
Suppose  $\la_{0}:  \D \longrightarrow [0, \infty]$.
By Proposition \ref{lapldtm} the function $\la$ generates deficient topological measure $\laplu$.
We claim that if each set in $\E$ has more than one point and $\la_{0} \neq 0$,  then  $\laplu$ is a 
deficient topological measure, but not a topological meausure. 
Indeed, taking a point from each  set in $\E$ we obtain a compact set $C$.  
It is easy to check that $\laplu (C) = \laplu(X \sm C) = 0$, but $ \laplu(X) > 0$. By Theorem \ref{DTMtoTM} 
$\laplu$ is not a topological measure.   

Note that if we repeat the same construction, starting with $\E$  as a finite family of disjoint unbounded connected subsets of $X$,
then we obtain $\la = 0$ on $ \kx$, and then $ \laplu = 0$. 

\item \label{exB}
Suppose  $\la_{0}:  \D \longrightarrow [-\infty, \infty]$ attains at most one of $ \infty, -\infty$. 
By Proposition \ref{lapldtm} the function $\la$ generates deficient topological measure $|\la |$.
As in part \ref{exA},  if each set in $\E$ has more than one point and $\la_{0} \neq 0$,  then  $| \la |$ is a 
deficient topological measure, but not a topological meausure. 
\end{enumerate}
\end{example}

We can say more about deficient topological measure $ \laplu$ in Example \ref{lanolDTM}, part \ref{exA} 
under additional assumptions on $\la_{0}$.   We first need two lemmas.

\begin{lemma} \label{ezle1}
Suppose $\{ E_i  \}$ is a finite family of subsets of $X, \ E \se X$ and $ E_i \not\subset E$ for each $i$. 
Then there is an open set $U$ such that  $ E \se U , \ \ E_i \not\subset  U$ for each $i$.
\end{lemma}

\begin{proof}
The proof as in \cite{Svistula:DTM} (Lemma 1):
construct a closed set $C$ by taking one point in each $E_i \sm E$. Take $ U = X \sm C$.
\end{proof}

\begin{lemma} \label{ezle2}
Let $X$ be locally compact. Assume that $K$ is a component of a compact set $C$  and $ K \se U, \ U \in \ox$. 
Then there are disjoint compact sets $E, D$ such that $E \sc D = C$ and $ K \se E \se U$.
\end{lemma}

\begin{proof}
We consider $C$ as a subspace of $X$. $C$ is compact, and we know (see, for example,
\cite{Engelking}, paragraph 6.1.23) that its component $K$ is a quasi-component, i.e. 
the intersection of all clopen subsets of $C$ containing $K$. Using, for example, paragraph 3.1.5 in \cite{Engelking}, 
it is easy to obtain a clopen subset $E$ of $C$ such that $K \se E \se U$.  Take $D = C \sm E$.
\end{proof}   
 
\begin{proposition} \label{extoflan}
Let $X$ be locally compact. Suppose $\la_{0}$ is a nonnegative set function on $\kx \cup \ox$, which is
monotone on $\ox$,  finitely additive on $\kx$, and outer regular on $ \kx$. 
Let $\la$ be a set function on $\kx$ defined as in (\ref{lala0}) in Example \ref{lanolDTM} 
for some finite family $\E$ of disjoint bounded nonempty connected subsets of $X$ 
and $\laplu$ defined as in Definition \ref{laplu}. Then 
$\laplu(K) = \la(K) $ for all $ K \in \kx$. Thus, if $X$ is compact, then $\la (C) = \laplu(C)$ for all $ C \in \cx$.  
\end{proposition}

\begin{proof}
The result follows from Proposition \ref{reg4ext}  applied to $\la$, if we check condition (\ref{regusl}).
Let $C$ be compact, $ \la(C) < \infty$. If $C$  contains no sets from $\E$, 
then by Lemma \ref{ezle1} there exists an open set $U$ 
such that $C \se U$ and $U$ does not contain elements of $\E$. Then $\la(C) = 0= \la(K)$ for any compact $K \se U$, 
and $\la$ satisfies condition (\ref{regusl}). 
Now assume that $C$ contains sets from $\E$. Let $\eps >0$. 
Recall that $\la(C) = \sum_{i=1}^n \la_{0}(C_i)$ where $C_i$ are components of $C$ that 
contain sets from $\E$. 
Since the sets $C_i$ s are disjoint and $\la_{0}$ is monotone on open sets and outer regular, we may find disjoint
open sets $U_i, \ i =1, \ldots, n$ such that $C_i \se U_i$ and $ \la_{0}(U_i) < \la_{0}(C_i) + \frac1n$. By Lemma \ref{ezle2}
for each $i$ find disjoint compact sets $E_i, D_i$ such that $E_i \sc D_i = C$, and $C_i \se E_i \se U_i$. 
Let $D = \bigcap_{i=1}^n D_i$. Then $D$ is compact, $E_i \cap D = \O$ for each $i$, 
and $C= \bigcap_{i=1}^n (E_i \sc D_i) = D \sc E_1 \sc \ldots \sc E_n$.
Since $D, E_1, \ldots, E_n$ are disjoint, choose disjoint open sets $V, V_1, \ldots, V_n$ such that 
$D \se V$ and $ E_i \se V_i \se U_i$ for $i=1, \ldots, n$. 
Note that $D \cap E_i = \O$ implies $D \cap C_i = \O$ for each $i$, so 
$D$, being a subset of $C$, does not contain any element from $\E$. By Lemma \ref{ezle1} we may assume that $V$ 
also does not contain elements of $\E$.  Let $W = V \sc V_1 \sc \ldots \sc V_n$. We shall show that $W$ is the set needed 
in condition (\ref{regusl}). So let $K \se W$ be compact, $K_1, \ldots, K_m$ be its components that 
contain sets from $\E$.
By connectedness each $K_j \se V_i$ for some $i$. Set $\Gamma_i = \{ j : K_j \se V_i\}$ for $i=1, \ldots, n$, 
and consider nonempty $\Gamma_i$. 
By outer regularity and finite additivity of $\la_{0}$  on $\kx$ we have: $\la_{0}( \bsc_{j \in \Gamma_i} K_j ) \le \la_0(V_i)$.
Now 
\begin{align*}
\la(K) &= \sum_{j=1}^m  \la_{0} (K_j) = \sum_{i:  \Gamma_i \neq \O} \la_{0} ( \bsc_{j \in \Gamma_i} K_j) \le \sum_{i=1}^n \la_0(V_i) \\
&\le \sum_{i=1}^n \la_0(U_i) < \sum_{i=1}^n \la_0(C_i) + \eps = \la(C) + \eps.
\end{align*}
Thus, condition (\ref{regusl})  is satisfied for $\la$, and this finishes the proof.
\end{proof}

\begin{remark}
Lemma \ref{ezle2} and Proposition \ref{extoflan} are generalizations of Lemma 2 and 
part of an argument on p. 733 in \cite{Svistula:DTM}. 
\end{remark} 

\begin{example}
Let $\la_{0}$ be a deficient topological measure. 
Let $K \in \kx$ and let $ \D = \{K\}$.  We apply Proposition \ref{extoflan}. For a compact $C$ disjoint from $K$ we see that
$\la(C) = 0$, whether or not $ \la_0 (C) = 0$. Thus, in general $ \laplu \neq \la_0$. 
\end{example}

\begin{example} \label{1ptRnDTM}
Let  $\la_{0}$ be a non-trivial deficient topological measure on a locally compact  space $X$. 
We may have $\la_{0} (X) < \infty $ or $\la_{0} (X) = \infty $.  Suppose there is a point $z$ in $X$ for which $\la_{0}(\{z\}) =0$. 
(For example, we may take $\la_{0}$ to be the Lebesque measure or a point mass $\delta_y$ at $y \ne z$ on $\r^n$.
We may also take $\la_{0}$ to be any topological measure from the last section in \cite{Butler:TMLCconstr}.)
Let the family $\E$ consist of one set, $ K = \{ z\}$.
Defined as in Example \ref{lanolDTM} the set function $ \la$ generates a deficient topological measure $\laplu$, and by 
Proposition \ref{extoflan} $\laplu = \la$ on $\kx$.
Then $\laplu(K)  = \la(K) = \la_{0} (K) = 0$. For any compact $C \se X \sm K$ we have $ \laplu(C) = \la(C) = 0$,
and then $ \laplu(X \sm K) = 0$. Since $\laplu(X) >0$, by Theorem \ref{DTMtoTM}, $\laplu$ is not a topological measure. 
Note that if, for instance,  $\la_{0}$ is the Lebesque measure and $X = \r^n$, 
then the range of deficient topological measure $ \laplu$
is $[0, \infty]$. If $\la_0$ is a compact-finite topological measure, then $\laplu$ is also compact-finite. 
\end{example}

\begin{remark} 
Example \ref{1ptRnDTM} is easy to generalize to the case when $\E$ consists of one set, $ K = \{ z_1, \ldots, z_n \} \se X$ 
for which $\la_{0}(K) =0$. 
\end{remark}

\begin{example} \label{nptsRnDTM}
Suppose $\la_{0}$ is a deficient topological measure on a locally compact  space $X$. 
Let the family $\E= \{ \{z_1\},  \ldots, \{z_n \} \} $,  $ K = \{ z_1, \ldots, z_n \} \se X$, and  $\la_{0}(K) =0$. 
For example, we may take $\la_{0}$ to be a Lebesque measure or a point mass  $\delta_y$  at $ y \ne z_1, \ldots, z_n$ on $\r^n$,
or any topological measure from the last section in \cite{Butler:TMLCconstr}. 
The set function $ \la$ as in Example \ref{lanolDTM} generates a deficient topological measure $\laplu$, and by 
Proposition \ref{extoflan} $\laplu = \la$ on $\kx$. As in Example \ref{1ptRnDTM} we have
$\laplu(K)  =  \laplu(X \sm K) = 0$, while  $\laplu(X) >0$. By Theorem \ref{DTMtoTM}, $\laplu$ is not a topological measure. 
Again, the range of $\laplu$ could be $[0, \infty]$. 
\end{example}

\begin{remark}
In Example \ref{1ptRnDTM}  and Example \ref{nptsRnDTM} $\laplu$ is not a topological measure 
even though no set in $\E$ contains more than 1 point. Compare to Example \ref{lanolDTM}. 
\end{remark}

If $X$ is also locally connected we may strengthen Proposition \ref{extoflan} and specify how $ \laplu$ acts 
on compact and open sets.

\begin{proposition} \label{lanolgen}
Let $X$ be locally compact and locally connected. 
Suppose $\la_{0}$ is a nonnegative set function on $\kx \cup \ox$, which is
monotone on $\ox$,  finitely additive on $\kx$, and outer regular on $ \kx$.  
Suppose $\la$ is a set function on $\kx$ defined as in (\ref{lala0}) in Example \ref{lanolDTM}
for some finite family $\E$ of disjoint bounded nonempty connected subsets of $X$, 
and $\laplu$ is defined as in Definition \ref{laplu}. 
Then for $A \in \kx \cup \ox$
\[ \laplu(A) = \sum_{i=1}^n \la_{0} (A_i), \]
where $A_i$ are the components of $A$ containing at least one set from $\E$; if there are
no such components,  then $\laplu(A) = 0$.
\end{proposition}

\begin{proof}
For $A$  compact, the statement follows from Proposition \ref{extoflan}. Now let $U$ be open. If $U$ does not contain
any sets from $\E$ then for any compact $C \se U$ by Proposition \ref{extoflan} we have $\laplu(C) = \la(C) =0$, 
and hence, by inner regularity $\laplu(U) = 0$. Now assume that $U_i, i=1, \ldots, n$ are the components of $U$  that contain 
at least one set from $\E$. Let $C  \in \kx, C \se U$, and $C_1, \ldots, C_m$ be components of $C$ that   
contain at least one set from $\E$. By connectedness, each $C_j \se U_i$ for some $i$. 
Set $\Gamma_i = \{ j : C_j \se U_i\}$ for $i=1, \ldots, n$, and consider nonempty $\Gamma_i$. 
By outer regularity of $\la_0$ we have: $\la_{0}( \bsc_{j \in \Gamma_i} C_j ) \le \la_{0} (U_i)$. 
Then $\la(C) = \sum_{j=1}^m \la_{0}(C_j) \le \sum_{i=1}^n \la_{0}(U_i),$ which implies that
\[ \laplu(U) \le \sum_{i=1}^n \la_{0}(U_i).\] 
We shall show the equality. It is enough to consider the case $\la_{0} (U_i) < \infty$ for each $i$, 
for if $\la_{0} (U_i) = \infty$ for some $i$ then for each $n \in \N$ there is a compact $K_n \se U_i \se U$ such that
$\la(K_n) > n$, and then $\laplu(U) = \infty$, so equality holds.
Let $\eps>0$. Since  each $U_i$ is open connected, 
and $\la_{0}$ is inner regular and monotone,  
using Lemma \ref{LeConLC} we may find a compact connected set $K_i \se U_i$ such that it 
contains elements of $\E$ that are contained in $U_i$ and $ \la_{0}(K_i) > \la_{0}(U_i) - \frac{\eps}{n}$.
Then for compact $K =\bsc_{i=1}^n K_i$ we have: $K \se U$ and 
\[ \laplu(U) \ge \la(K) = \sum_{i=1}^n \la_{0} (K_i)  >   \sum_{i=1}^n \la_{0} (U_i) - \eps. \]
It follows that $ \laplu(U) = \sum_{i=1}^n \la_{0}(U_i)$.
\end{proof}

\begin{remark}
Examples \ref{Svconset}, part \ref{exA} of Example \ref{lanolDTM}, and Example \ref{1ptRnDTM} 
are generalizations of Example 1, Example 2, and
Example (c) on p. 733 in \cite{Svistula:DTM}.
\end{remark}

\begin{example}  \label{discrDTM} 
Let $J$ be a countable subset of $X$. Let $\nu$ be a finitely additive set function on $\calP (J)$ with $ \nu(\O) =0$ 
that does not assume both $\infty$ and
$-\infty$. Define a set function on $\kx$ as follows: if $K \cap J \neq \O$ then
\[ \la(K) = \sum_{i=1}^ \infty  \nu(K_i \cap J), \]
where $K$ is compact, and $K_i$ are components of $K$ that intersect $J$. If $K \cap J = \O$ then $\la(K) = 0$.
Again,  $\la$ is finitely additive on $\kx$. By Proposition \ref{lapldtm}, $\laplu$ is a deficient topological measure on $X$.
\begin{enumerate}[label=(\roman*),ref=(\roman*)]
\item  \label{firpa} 
Let $\nu(L) = |L| $ for a finite $L \in \calP (J)$,  and  $\nu(L) = \infty$ otherwise.  Then  $\laplu(A)$ counts how many
points from $J$ are contained in $A$.  
\item \label{secpa}
If $\E$ is a finite family of disjoint connected subsets of $X$ 
(as in Example \ref{lanolDTM}), $J$ is a set where we take one point from each set in $\E$, and 
$\nu(L) = |L|$ for $L \in \calP (J)$,  then  $\laplu(A)$ counts how many
sets from $\E$ are contained in $A$. 
\item \label{semifinDTM}
Let $X$ be $\r^2$  or $X = [-10, 10] \times   [-10, 10] $. Set $J = J_1 \cup J_2$, 
where $J_1$ is the set $ \{ (0, \frac 1n), \ n \in \N \}$ and 
$J_2$ is the set of points with all natural coordinates, and let $\nu$ be as in part \ref{firpa}. 
Then $\laplu (C) = \infty$ for any compact 
$C$ containing $J_1$,  and  $\laplu (C) < \infty$ for any compact 
$C$ not intersecting $J_1$. It is easy to see that $\laplu$ is not compact-finite, but is semifinite. 
\end{enumerate}
\end{example}  

\begin{remark}
Part \ref{secpa} in Example \ref{discrDTM} implies Example (a2) on p. 732 in \cite{Svistula:DTM}.
\end{remark}

{\bf{Acknowledgments}}:
This  work was conducted at the Department of Mathematics at the University of California Santa Barbara. 
The author would like to thank the department for its hospitality and supportive environment.
%

${   }$ \\
\noindent
Department of Mathematics \\
University of California Santa Barbara \\
552 University Rd.,  \\
Isla Vista, CA 93117, USA \\
e-mail: svbutler@ucsb.edu \\


\end{document}